\RequirePackage{fix-cm}
\documentclass{svjour3}                     
\smartqed  
\usepackage{graphicx}

\usepackage{color}
\usepackage{algorithm}
\usepackage{algpseudocode}
\usepackage{caption}
\usepackage{tabto}
\usepackage{ntheorem}
\usepackage[toc,page]{appendix}
\usepackage{enumitem}
\newcommand{\myparagraph}[1]{\paragraph{#1}\mbox{}\\}
\usepackage{amsmath,amssymb}

\usepackage[a4paper,margin=1in]{geometry}

 \DeclareMathOperator*{\argmax}{argmax} 
 
 \newcommand{\operator}[1]{{\normalfont \texttt{#1}}}

\newtheorem*{complexity}{\normalfont \textit{Complexity}}

\begin{document}

\title{Representation of Polytopes as Polynomial Zonotopes
}


\author{Niklas Kochdumper        \and
        Matthias Althoff 
}


\institute{Niklas Kochdumper \at
              Technical University of Munich\\
              \email{niklas.kochdumper@tum.de}           
           \and
           Matthias Althoff\at
              Technical University of Munich\\
              \email{althoff@tum.de}   
}

\date{}

\maketitle

\begin{abstract}
We prove that each bounded polytope can be represented as a polynomial zonotope, which we refer to as the Z-representation of polytopes. Previous representations are the vertex representation (V-representation) and the halfspace representation (H-representation). Depending on the polytope, the Z-representation can be more compact than the V-representation and the H-representation. In addition, the Z-representation enables the computation of linear maps, Minkowski addition, and convex hull with a computational complexity that is polynomial in the representation size. The usefulness of the new representation is demonstrated by range bounding within polytopes.
\end{abstract}


\section{Introduction}
\label{sec:Introduction}

\subsection{Related Work}
\label{subsec:RelatedWork}

The two main representations of polytopes are the vertex representation (V-representation) and the halfspace representation (H-representation). While the V-representation describes a polytope as the convex hull of its vertices, the H-representation defines a polytope as the intersection of halfspaces. Whether the V-representation or the H-representation is more compact depends on the polytope. The conversion from V-representation to H-representation is known as the \textit{facet enumeration problem}, and the inverse problem of finding the V-representation given the H-representation is known as the \textit{vertex enumeration problem}. Both transformations can be computed in polynomial time with respect to the number of polytope vertices \cite{Kaibel2003}. One of the major criteria for a comparison of different set representations is the computational efficiency of operations on them. Of special interest are operations that are closed for polytopes: Linear transformation, Minkowski sum, convex hull, and intersection. 

Let us first consider the H-representation. Linear transformations are simple if the matrix $M$ defining the transformation is invertible; otherwise, the computational complexity is exponential in the dimension $n$ \cite{Jones2008}. The complexity of computing the Minkowski sum of two H-polytopes is also exponential in $n$ since the worst-case number of resulting halfspaces is exponential in $n$ \cite{Hagemann2015,Tiwary2008}. The calculation of the convex hull of two polytopes in H-representation is NP-hard \cite{Tiwary2008}. Computation of the intersection of two polytopes in H-representation can be easily realized by a concatenation of the inequality constraints. A possible subsequent elimination of redundant halfspaces can be implemented with linear programming. 

For the V-representation, the situation is different: The linear transformation is trivial to calculate, even for cases where the transformation matrix $M$ is not invertible. Computation of the Minkowski sum of two polytopes in V-representation has exponential complexity in the number of dimensions \cite{Fukuda2004,Gritzmann1993}. The same holds for the convex hull operation, because constructing a redundant V-representation of the convex hull is trivial, and the redundant points can be eliminated by solving an exponential number of linear programs \cite{Tiwary2008}. However, computation of the intersection of two polytopes in V-representation is NP-hard \cite{Tiwary2008}. Overall, neither the H-representation nor the V-representation has polynomial complexity for all four discussed operations. 

Two other set representations which are able to represent any bounded polytope are constrained zonotopes \cite{Scott2016} and zonotope bundles \cite{Althoff2011f}. Constrained zonotopes are zonotopes with additional linear equality constraints on the zonotope factors. Since the implementation of the linear transformation, the Minkowski sum, and the intersection operation only involves matrix multiplications, additions, and concatenations \cite[Prop. 1]{Scott2016}, their computational complexity is at most $\mathcal{O}(n^3)$. Furthermore, there exist efficient techniques for the removal of redundant constraints and zonotope generators \cite[Sec. 4.2]{Scott2016}. However, at present there is no known algorithm for the computation of the convex hull of two constraint zonotopes. Zonotope bundles define a set implicitly as the intersection of multiple zonotopes. While the computation of the linear transformation and the intersection is trivial, efficient algorithms for the Minkowski sum \cite[Prop. 2]{Althoff2011f} and the convex hull \cite[Prop. 5]{Althoff2011f} only exist for the calculation of an over-approximation. 
 
The Z-representation that we present in this paper is based on polynomial zonotopes, a non-convex set representation first introduced in \cite{Althoff2013a}. In particular, we use the sparse representation of polynomial zonotopes from \cite{Kochdumper2019}. Taylor Models are a type of set representation closely related to polynomial zonotopes \cite{Makino2003}. However, while Taylor models, like polynomial zonotopes, are able to represent non-convex sets, they can not represent arbitrary polytopes.  

The Z-representation is based on polynomials; the interplay between polytopes and polynomials has a long history in computational geometry: the Newton polytope of a polynomial is the convex hull of its exponent vectors \cite{Khovanskii1977}. Newton polytopes are, among other things, useful for analyzing the roots of multivariate polynomials \cite{Sturmfels1998}. The Ehrhart polynomial of a polytope specifies the number of integer points the polytope contains \cite{Hibi1990}. Furthermore, it is well-known that polytopes can be equivalently represented by polynomial inequalities \cite{Grotschel2003}. The classical result of Minkowski implies that a polytope can be approximated arbitrary closely by a single polynomial inequality \cite{Minkowski1903}. 

In this work, we introduce the Z-representation for bounded polytopes. This new representation enables the computation of the linear transformation, Minkowski sum, and convex hull with polynomial complexity with respect to the representation size. In addition, the Z-representation can be more compact than the V- and the H-representation. We further provide algorithms for the conversion from V-representation to Z-representation and from Z-representation to V-representation. One application of the Z-representation is range bounding, for which we demonstrate the advantages resulting from the new representation.

\subsection{Notation}
\label{subsec:Notation}

In the remainder of this paper, we will use the following notations: Sets and tuples are always denoted by calligraphic letters, matrices by uppercase letters, and vectors by lowercase letters. Given a vector $b \in \mathbb{R}^n$, $b_{(i)}$ refers to the $i$-th entry. Given a matrix $A \in \mathbb{R}^{n \times m}$, $A_{(i,\cdot)}$ represents the $i$-th matrix row, $A_{(\cdot,j)}$ the $j$-th column, and $A_{(i,j)}$ the $j$-th entry of matrix row $i$. The empty matrix is denoted by $[~]$. Given two matrices $C$ and $D$, $[C,D]$ denotes the concatenation of the matrices. The symbols $\mathbf{0}$ and $\mathbf{1}$ represent matrices of zeros and ones of proper dimension. Given a $n$-tuple $\mathcal{L} = (l_1,\dots,l_n)$, $|\mathcal{L}| = n$ denotes the cardinality of the tuple and $\mathcal{L}_{(i)}$ refers to the $i$-th entry of tuple $\mathcal{L}$. Given a tuple $\mathcal{H} = (h_1,\dots,h_n)$ with $h_i \in \mathbb{R}^{m_i} ~ \forall i \in \{1,\dots,n\}$, notation $\mathcal{H}_{(i,j)} = h_{i(j)}$ refers to the $j$-th entry in the $i$-th element of $\mathcal{H}$. The empty tuple is denoted by $\emptyset$. Given two tuples $\mathcal{L} = (l_1, \dots ,l_n)$ and $\mathcal{K} = (k_1, \dots, k_m)$, $( \mathcal{L}, ~ \mathcal{K} ) = (l_1, \dots ,l_n,k_1,\dots,k_m)$ denotes the concatenation of the tuples. We further introduce an n-dimensional interval as $\mathcal{I} := [l,u],~ \forall i ~ l_{(i)} \leq u_{(i)},~ l,u \in \mathbb{R}^n$. The ceil operator $\lceil x \rceil$ and the floor operator $\lfloor x \rfloor$ round a scalar number $x \in \mathbb{R}$ to the next higher and lower integer, respectively. The binomial coefficient is denoted by ${ r \choose z }$, $r,z \in \mathbb{N}$. For the derivation of computational complexity, we consider all binary operations, except concatenations and initializations since their computational time is negligible.


\section{Preliminaries}
\label{sec:Preliminaries}

We first provide some definitions that are important for the remainder of the paper.

\begin{definition}
	(V-Representation) Given the $q$ polytope vertices $v_i \in \mathbb{R}^n$, the vertex representation of $\mathcal{P} \subset \mathbb{R}^n$ is defined as
	\begin{equation*}
		 \mathcal{P} := \left\{ \left. \sum_{i=1}^{q} \beta_i v_i ~ \right |~ \beta_i \geq 0, ~ \sum_{i=1}^q \beta_i = 1 \right\}.
	\end{equation*}
	\label{def:Vrep}
\end{definition}
For a concise notation we introduce the shorthand $\mathcal{P} = \langle [v_1,\dots,v_q] \rangle _{V}$ for the V-representation. The V-representation is called redundant if the matrix $[v_1,\dots,v_q]$ contains points that are not vertices of $\mathcal{P}$.

\begin{definition}
	(H-Representation) Given a matrix $A \in \mathbb{R}^{h \times n}$ and a vector $b \in \mathbb{R}^h$, the halfspace representation of $\mathcal{P} \subset \mathbb{R}^n$ is defined as 
	\begin{equation*}
		 \mathcal{P} := \left\{ x ~|~ A x \leq b \right\}.
	\end{equation*}
\end{definition}

Polynomial zonotopes are a non-convex set representation that was first introduced in \cite{Althoff2013a}. We use a slight modification of the sparse representation of polynomial zonotopes from \cite{Kochdumper2019}:

\begin{definition}
  (Polynomial Zonotope) Given a generator matrix $G \in \mathbb{R}^{n \times h}$ and an exponent matrix $E \in \mathbb{Z}_{\ge 0}^{p \times h}$, a polynomial zonotope $\mathcal{PZ} \subset \mathbb{R}^n$ is defined as  
  \begin{equation}
    \mathcal{PZ} := \left\{ \left. \sum _{i=1}^h \left( \prod _{k=1}^p \alpha _k ^{E_{(k,i)}} \right) G_{(\cdot,i)} ~ \right| ~\alpha_k \in [-1,1]\right\}.
  \label{eq:polyZonotope}
  \end{equation}
  \label{def:polyZonotope}
\end{definition} 
The variables $\alpha_k$ are called factors, and the vectors $G_{(\cdot,i)}$ generators. The number of factors $\alpha_k$ is $p$, and the number of generators is $h$.

The set operations considered in this paper are linear transformation, Minkowski sum, convex hull, and intersection.  Given two sets $\mathcal{S}_1, \mathcal{S}_2 \subset \mathbb{R}^n$, these operations are defined as follows: 
\begin{alignat}{2}
	& \mathrm{- Linear}~\mathrm{transformation:} ~~ && M \mathcal{S}_1 = \{ M s_1 ~|~ s_1 \in \mathcal{S}_1 \},~ M \in \mathbb{R}^{m \times n} \label{eq:linTrans}\\
	& \mathrm{- Minkowski} ~\mathrm{sum:} && \mathcal{S}_1 \oplus \mathcal{S}_2 = \{ s_1 + s_2 ~|~ s_1 \in \mathcal{S}_1,~ s_2 \in \mathcal{S}_2 \} 
	\label{eq:minSum} \\
	& \mathrm{- Convex} ~ \mathrm{hull:} && conv(\mathcal{S}_1,\mathcal{S}_2) = \left \{ \frac{1}{2} (1+\lambda) s_1 + \frac{1}{2}(1-\lambda) s_2 ~|~ s_1 \in \mathcal{S}_1,~s_2 \in \mathcal{S}_2,~\lambda \in [-1,1] \right\} \label{eq:convHull} \\
	& \mathrm{- Intersection:} && \mathcal{S}_1 \cap \mathcal{S}_2 = \{ s ~|~ s \in \mathcal{S}_1,~ s \in \mathcal{S}_2 \} \label{eq:intersection}
\end{alignat}


\section{Z-Representation of Polytopes}
\label{sec:Z-Representation}

In this section, we introduce the Z-representation of bounded polytopes, we derive the equations for set operations applied in Z-representation, and we provide algorithms for the conversion between P- and V-representation. 

\subsection{Definition}

We first introduce the Z-representation of bounded polytopes:

\begin{definition}
	(Z-Representation) Given a starting point $c \in \mathbb{R}^n$ and a generator matrix $G \in \mathbb{R}^{n \times h}$ , the Z-representation defines the set
	\begin{equation}
		\begin{split}
		&\mathcal{P} := \left\{ \left. c + \sum _{i=1}^h \left( \prod _{k=1}^{m_i} \alpha _{\mathcal{E}_{(i,k)}} \right) G_{(\cdot,i)} ~ \right| ~\alpha_{\mathcal{E}_{(i,k)}} \in [-1,1] \right\} \\
		& ~ \\
		&\mathcal{E} = \left( e_1, \dots ,e_h \right),~\forall i \in \{1,\dots,h \}: ~ e_i \in \mathbb{N}_{\leq p}^{m_i},~ \mathrm{and}~  \forall i \in \{1,\dots,h \} ~\forall j,k: ~ j \neq k \Rightarrow e_{i(j)} \neq e_{i(k)} ,
		\end{split}
		\label{eq:pRep}
	\end{equation}
	where the tuple $\mathcal{E}$ stores the factor indices, $m_i$ is the length of the $i$-th element of $\mathcal{E}$, $p$ is the number of factors $\alpha _{\mathcal{E}_{(i,k)}}$, and $h$ is the number of generators. The overall number of entries in $\mathcal{E}$ is 
	\begin{equation}
	 	\mu = \sum_{i=1}^h m_i.
	 	\label{eq:numListElements}
	 \end{equation} 
The Z-representation is regular if
	\begin{equation}
		\forall j,k:~ j \neq k \Rightarrow e_k \neq e_j.
		\label{eq:regular}
	\end{equation}
	\label{def:pRep}
\end{definition}
For a concise notation we introduce the shorthand $\mathcal{P} = \langle c,G, \mathcal{E} \rangle_Z$. All components of a set $\square_i$ have the index $i$, e.g., $p_i$, $h_i$, $\mu_i$, and $m_{i,j}$, $j=1\dots h_i$ belong to $\mathcal{P}_i = \langle c_i,G_i,\mathcal{E}_i \rangle_Z$. In the remainder of this work, we call the term $\alpha_{\mathcal{E}_{(i,1)}} \cdot \ldots \cdot \alpha_{\mathcal{E}_{(i,m_i)}} \cdot G_{(\cdot,i)}$ in \eqref{eq:pRep} a monomial and $\alpha_{\mathcal{E}_{(i,1)}} \cdot \ldots \cdot \alpha_{\mathcal{E}_{(i,m_i)}}$ the variable part of the monomial. The Z-representation defines a special type of polynomial zonotope (see Def. \ref{def:polyZonotope}) where the exponents of the factors $\alpha_k$ are restricted to the values 0 and 1. Therefore, every set in Z-representation can be equivalently represented as a polynomial zonotope. It further holds that every bounded polytope can be equivalently represented by the Z-representation, which we prove later in Theorem \ref{theo:mainTheorem}. The converse does not hold: not every set defined by a Z-representation is a bounded polytope. We illustrate the Z-representation of polytopes with two simple examples:

\begin{example}
	The Z-representation
	\begin{equation*}
		\mathcal{P} = \left\langle \begin{bmatrix} -0.5 \\ 0 \end{bmatrix}, \begin{bmatrix} 1.5 & -0.5 & -0.5 \\ -0.5 & -2 & 0.5\end{bmatrix},~ \left( 1, 2, \begin{bmatrix} 1 \\ 2 \end{bmatrix} \right) \right\rangle_Z
	\end{equation*}
	defines the polytope 
	\begin{equation*}
		\mathcal{P} = \left\{ \left. \begin{bmatrix} -0.5 \\ 0 \end{bmatrix} + \alpha_1 \begin{bmatrix} 1.5 \\ -0.5 \end{bmatrix} + \alpha_2 \begin{bmatrix} -0.5 \\ -2 \end{bmatrix} + \alpha_1 \alpha_2 \begin{bmatrix} -0.5 \\ 0.5 \end{bmatrix} ~ \right| ~ \alpha_1, \alpha_2 \in [-1,1]  \right\},
	\end{equation*}
	\label{ex:pRep1}
\end{example}
which is visualized in Fig. \ref{fig:Example1and2} (left).

\begin{example}
	The Z-representation
	\begin{equation*}
		\mathcal{P} = \left\langle \begin{bmatrix} -0.5 \\ 0 \end{bmatrix},\begin{bmatrix} -0.5 & -0.5 & 1.5 \\ 0.5 & -2 & -0.5 \end{bmatrix},~ \left( 1, 2, \begin{bmatrix} 1 \\ 2 \end{bmatrix} \right) \right\rangle_Z
	\end{equation*}
	defines the set
	\begin{equation*}
		\mathcal{P} = \left\{ \left. \begin{bmatrix} -0.5 \\ 0 \end{bmatrix} + \alpha_1 \begin{bmatrix} -0.5 \\ 0.5 \end{bmatrix} + \alpha_2 \begin{bmatrix} -0.5 \\ -2 \end{bmatrix} + \alpha_1 \alpha_2 \begin{bmatrix} 1.5 \\ -0.5 \end{bmatrix} ~ \right| ~ \alpha_1, \alpha_2 \in [-1,1]  \right\},
	\end{equation*}
	which is not a polytope as can be seen in Fig. \ref{fig:Example1and2} (right).
	\label{ex:pRep2}
\end{example}

\begin{figure}
	\centering
  	\includegraphics[width = 0.95 \textwidth]{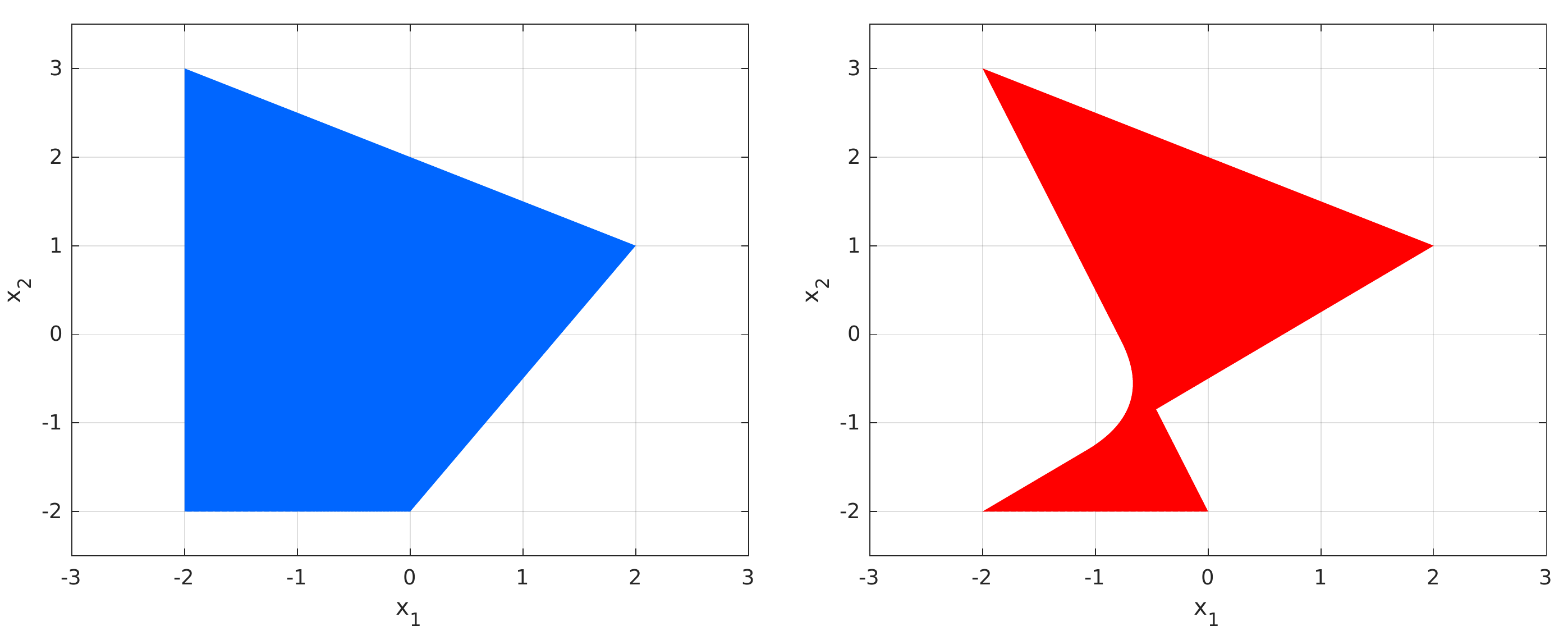}
	\captionof{figure}{Visualization of the set defined by the Z-representation from Example \ref{ex:pRep1} (left) and from Example \ref{ex:pRep2} (right).}
	\label{fig:Example1and2}
\end{figure}

\subsection{Operations}

In this subsection, we derive and prove implementations for the set operations in \eqref{eq:linTrans}-\eqref{eq:intersection}. In addition, we determine the computational complexity for each operation. For further derivations, let us establish that according to the definition of the Z-representation in \eqref{eq:pRep} we can write
\begin{equation}
	\begin{split}
		& \left \{ c + \sum _{i=1}^{h_1} \left( \prod _{k=1}^{m_{1,i}} \alpha _{\mathcal{E}_{1(i,k)}} \right) G_{1(\cdot,i)} + \sum _{i=1}^{h_2} \left( \prod _{k=1}^{m_{2,i}} \alpha _{\mathcal{E}_{2(i,k)}} \right) G_{2(\cdot,i)} ~ \bigg | ~ \alpha _{\mathcal{E}_{1(i,k)}}, \alpha _{\mathcal{E}_{2(i,k)}} \in [-1,1] \right \} \\ 
		& ~ \\
		&  = \left \langle c,[G_1, G_2], \big( \mathcal{E}_1, \mathcal{E}_2 \big)  \right \rangle_Z.
	\end{split}	
	\label{eq:sumPrep}
\end{equation}

\subsubsection{Linear Transformation}

\begin{proposition}
	(Linear Transformation) Given a set in Z-representation $\mathcal{P} = \langle c, G, \mathcal{E} \rangle_Z \subset \mathbb{R}^n$ and a matrix $M \in \mathbb{R}^{m \times n}$, the linear transformation is computed as 
	\begin{equation}
		M \mathcal{P} = \langle M  c, ~M G,~ \mathcal{E}  \rangle_Z,
	\end{equation}
	which has complexity $\mathcal{O}(mnh)$, where $h$ is the number of generators.
\end{proposition}
\begin{proof}
	The result follows directly from inserting the definition of the Z-representation in \eqref{eq:pRep} into the definition of the linear transformation \eqref{eq:linTrans}. 
\end{proof} 

\begin{complexity} \normalfont The complexity results from the matrix multiplication with the starting point $c$ and the generator matrix $G$ and is therefore $\mathcal{O}(mn) + \mathcal{O}(mnh) = \mathcal{O}(mnh)$. \hfill $\square$
\end{complexity}

\subsubsection{Minkowski Sum}

\begin{proposition}
	(Minkowski Sum) Given two sets in Z-representation $\mathcal{P}_1 = \langle c_1, G_1, \mathcal{E}_1 \rangle_Z$ and $\mathcal{P}_2 = \linebreak[4] \langle c_2, G_2,\mathcal{E}_2 \rangle_Z$, the Minkowski sum is computed as 
	\begin{equation}
		\begin{split}
			& \mathcal{P}_1 \oplus \mathcal{P}_2 = \left \langle c_1 + c_2, \left[ G_1 , G_2 \right], \left( \mathcal{E}_1, \overline{\mathcal{E}}_2 \right)	\right \rangle_Z \\
			& ~ \\
			& \mathrm{with} ~~ \overline{\mathcal{E}}_2 = \left( \mathcal{E}_{2(1)} + \mathbf{1} ~ p_1 ~, ~ \dots ~, ~ \mathcal{E}_{2(h_2)} + \mathbf{1} ~ p_1 \right),
		\end{split}
		\label{eq:minSumPrep}
	\end{equation}
	which has complexity $\mathcal{O}(n + \mu_2)$, where $p_1$ is the number of factors of $\mathcal{P}_1$ and $\mu_2$ is the number of entries in tuple $\mathcal{E}_2$.
\end{proposition}
\begin{proof}
	The proposition follows from inserting the definition of the Z-representation in \eqref{eq:pRep} into the definition of the Minkowski sum \eqref{eq:minSum} and using \eqref{eq:sumPrep}:
	\begin{equation}
	\begin{split}
		\mathcal{P}_1 \oplus \mathcal{P}_2 &= \Bigg \{ \underbrace{c_1 + c_2}_{c} + \sum _{i=1}^{h_1} \left( \prod _{k=1}^{m_{1,i}} \alpha _{\mathcal{E}_{1(i,k)}} \right) G_{1(\cdot,i)} + \sum _{i=1}^{h_2} \Bigg( \prod _{k=1}^{m_{2,i}} \underbrace{\alpha _{\mathcal{E}_{2(i,k)} + p_1}}_{\alpha_{\overline{\mathcal{E}}_{2(i,k)}}} \Bigg) G_{2(\cdot,i)} ~ \Bigg|~ \alpha _{\mathcal{E}_{1(i,k)}},~ \alpha_{\overline{\mathcal{E}}_{2(i,k)}}  \in [-1,1] \Bigg \} \\ 
		& ~ \\
		& \overset{\eqref{eq:sumPrep}}{=}  \left \langle c_1 + c_2, \left[ G_1 , G_2 \right], \left( \mathcal{E}_1, \overline{\mathcal{E}}_2 \right)	\right \rangle_Z.
	\end{split}
		\label{eq:proofMinSum}
	\end{equation}
\end{proof}
	
\begin{complexity} \normalfont The complexity for the addition of the center vectors is $\mathcal{O}(n)$, and the complexity for the construction of the tuple $\overline{\mathcal{E}}$ is $\mathcal{O}(\mu_2)$ with $\mu_2$ defined as in \eqref{eq:numListElements}. Thus, the overall complexity is $\mathcal{O}(n + \mu_2)$. \hfill $\square$
\end{complexity}

\subsubsection{Convex Hull} 

\begin{proposition}
	(Convex Hull) Given two sets in Z-representation $\mathcal{P}_1 = \langle c_1, G_1, \mathcal{E}_1 \rangle_Z$ and $\mathcal{P}_2 = \langle c_2, G_2,\mathcal{E}_2 \rangle_Z$, their convex hull is computed as 
	\begin{equation}
		\begin{split}
			& conv(\mathcal{P}_1,\mathcal{P}_2) = \left \langle \frac{1}{2} \left( c_1 + c_2 \right), \frac{1}{2} \left[ \left( c_1 - c_2 \right), G_1 , G_1 , G_2 , -G_2 \right], \left( ( p ), ~ \mathcal{E}_1, ~ \widehat{\mathcal{E}}_1, ~ \overline{\mathcal{E}}_2, ~ \widehat{\mathcal{E}}_2 \right)	\right \rangle_Z \\
			& ~ \\
			& \mathrm{with} ~~ \widehat{\mathcal{E}}_1 = \left( [ \mathcal{E}_{1(1)}^T , p ]^T ~,~ \dots ~,~ [ \mathcal{E}_{1(h_1)}^T , p ]^T  \right), \\
			& ~~~~~~~ \overline{\mathcal{E}}_2 = \left( \mathcal{E}_{2(1)} + \mathbf{1} ~ p_1 ~,~ \dots ~,~ \mathcal{E}_{2(h_2)} + \mathbf{1} ~ p_1  \right), \\
			& ~~~~~~~ \widehat{\mathcal{E}}_2 = \left( [ \overline{\mathcal{E}}_{2(1)}^T ,p ]^T ~,~ \dots ~,~ [ \overline{\mathcal{E}}_{2(h_2)}^T ,p ]^T  \right).
		\end{split}
		\label{eq:convexHullPrep}
	\end{equation}
	It holds that 
	\begin{equation}
		\begin{split}
			& p = p_1 + p_2 + 1 \\
			& h = 2 h_1 + 2 h_2 + 1 \\
			& \mu = 2 \mu_1 + 2 \mu_2 + h_1 + h_2 + 1,
		\end{split}
		\label{eq:changeProp}
	\end{equation}	
	where $p$ is the number of factors, $h$ the number of generators, and $\mu$ the number of tuple entries $($see \eqref{eq:numListElements}$)$ of $conv(\mathcal{P}_1,\mathcal{P}_2)$. The complexity is $\mathcal{O}(n + \mu_2)$, where $\mu_2$ is the number of entries in the tuple $\mathcal{E}_2$.
	\label{prop:convexHullPrep}
\end{proposition}

\begin{proof}
	The result follows from inserting the definition of the Z-representation in \eqref{eq:pRep} into the definition of the convex hull \eqref{eq:convHull} and using \eqref{eq:sumPrep}:
	\begin{equation}
	\begin{split}
		conv(\mathcal{P}_1 , \mathcal{P}_2) &= \Bigg\{ \underbrace{\frac{1}{2} (c_1 + c_2)}_{c} + \frac{1}{2} (c_1 - c_2) \lambda + \frac{1}{2} \sum_{i=1}^{h_1} \left( \prod _{k=1}^{m_{1,i}} \alpha _{\mathcal{E}_{1(i,k)}} \right) G_{1(\cdot,i)} ~ + \\
		& ~~~~~~~  \frac{1}{2} \sum_{i=1}^{h_1} \lambda \left( \prod _{k=1}^{m_{1,i}} \alpha _{\mathcal{E}_{1(i,k)}} \right) G_{1(\cdot,i)} + \frac{1}{2} \sum _{i=1}^{h_2} \Bigg( \prod _{k=1}^{m_{2,i}} \underbrace{\alpha _{\mathcal{E}_{2(i,k)} + p_1}}_{\alpha_{\overline{\mathcal{E}}_{2(i,k)}}} \Bigg) G_{2(\cdot,i)} ~ - \\
		& ~~~~~~~  \frac{1}{2} \sum_{i=1}^{h_2} \lambda \Bigg( \prod _{k=1}^{m_{2,i}} \underbrace{\alpha _{\mathcal{E}_{2(i,k)} + p_1}}_{\alpha_{\overline{\mathcal{E}}_{2(i,k)}}} \Bigg) G_{2(\cdot,i)} ~ \Bigg| ~ \alpha _{\mathcal{E}_{1(i,k)}} \in [-1,1],~\alpha_{\overline{\mathcal{E}}_{2(i,k)}} \in [-1,1], ~ \lambda \in [-1,1] \Bigg\} \\
		& ~ \\
		& \overset{\overset{\overset{\eqref{eq:sumPrep}}{\mathrm{and}}}{\alpha_p := \lambda}}{=} \left \langle \frac{1}{2} \left( c_1 + c_2 \right), \frac{1}{2} \left[ \left( c_1 - c_2 \right), G_1 , G_1 , G_2 , -G_2 \right], \left( ( p ), ~ \mathcal{E}_1, ~ \widehat{\mathcal{E}}_1, ~ \overline{\mathcal{E}}_2, ~ \widehat{\mathcal{E}}_2 \right)	\right \rangle_Z,
	\end{split}
	\label{eq:proofConvHullPrep}
	\end{equation}
	where $p = p_1 + p_2 + 1$ (see \eqref{eq:changeProp}). For the transformation in the last line of \eqref{eq:proofConvHullPrep}, we substitute $\lambda$ with an additional factor $\alpha_p$. With this substitution and $\hat{\mathcal{E}}_1$ and $\hat{\mathcal{E}}_2$ defined as in \eqref{eq:convexHullPrep}, it holds that
	\begin{equation*}
		\begin{split}
			& \lambda \prod_{k=1}^{m_{1,i}} \alpha_{\mathcal{E}_{1(i,k)}} = \alpha_p \prod_{k=1}^{m_{1,i}} \alpha_{\mathcal{E}_{1(i,k)}} = \prod_{k=1}^{\hat{m}_{1,i} = m_{1,i} + 1} \alpha_{\hat{\mathcal{E}}_{1(i,k)}}, \\
			& \lambda \prod_{k=1}^{m_{2,i}} \alpha_{\overline{\mathcal{E}}_{2(i,k)}} = \alpha_p \prod_{k=1}^{m_{2,i}} \alpha_{\overline{\mathcal{E}}_{2(i,k)}} = \prod_{k=1}^{\hat{m}_{2,i} = m_{2,i} + 1} \alpha_{\hat{\mathcal{E}}_{2(i,k)}}.
		\end{split}
	\end{equation*}
	Since $\lambda \in [-1,1]$ and $\alpha_p \in [-1,1]$, substituting $\lambda$ with $\alpha_p$ does not change the set. 
	
	The number of generators $h$ in \eqref{eq:changeProp} directly follows from the construction of the generator matrix $0.5[c_1-c_2,G_1,G_1,G_2,-G_2]$. The number of tuple entries $\mu$ in \eqref{eq:changeProp} results from the construction of the tuple $( ( p ), ~ \mathcal{E}_1, ~ \widehat{\mathcal{E}}_1, ~ \overline{\mathcal{E}}_2, ~ \widehat{\mathcal{E}}_2 )$, since $\mathcal{E}_1$ has $\mu_1$ entries, $\hat{\mathcal{E}}_1$ has $\mu_1 + h_1$ entries, $\overline{\mathcal{E}}_2$ has $\mu_2$ entries, and $\hat{\mathcal{E}}_2$ has $\mu_2 + h_2$ entries.
\end{proof}	
	
\begin{complexity} \normalfont The complexity for the addition and subtraction of the center vectors is $\mathcal{O}(2n)$, and the complexity for the construction of the set $\overline{\mathcal{E}}_2$ is $\mathcal{O}(\mu_2)$ with $\mu_2$ defined as in \eqref{eq:numListElements}. Thus, the overall complexity is $\mathcal{O}(n + \mu_2)$. \hfill $\square$
\end{complexity}

\subsubsection{Intersection}

Contrary to previous set operations where the computation in Z-representation is straightforward, the calculation of the intersection is non-trivial. At present, there exists no algorithm to compute the intersection directly in Z-representation without conversion to another polytope representation.

\subsection{Representation of Bounded Polytopes}
\label{subsec:RepBoundedPolytopes}

We now formulate and prove the main theorem of our paper:
\begin{theorem}
	Every bounded polytope can be equivalently represented in Z-representation.
	\label{theo:mainTheorem} 
\end{theorem}

\begin{proof}
	If the polytope is bounded, then the set can be described as the convex hull of its vertices (see Def. \ref{def:Vrep}). Each vertex $v_i \in \mathbb{R}^n$ can be equivalently represented by the Z-representation $v_i = \langle v_i, [~], \emptyset \rangle_Z$. Since the Z-representation is closed under the convex hull operation as shown in Prop. \ref{prop:convexHullPrep}, computation of the convex hull results in a Z-representation of the polytope. \hfill $\square$
\end{proof}
An algorithm for the conversion of a polytope in V-representation to a polytope in Z-representation is provided in the next subsection.

\subsection{Conversion to Z-Representation}
\label{subsec:Vrep2Prep}

Since none of the existing polytope representations permits efficient computation for all basic set operations, efficient algorithms for the conversion between different representations are important. We first provide an algorithm that converts a polytope from V-representation to Z-representation:

\begin{algorithm}[H]
	\caption{Conversion from V-representation to Z-representation} \label{alg:Vrep2Prep}
	\textbf{Require:} Bounded polytope in V-representation $\mathcal{P} = \langle [v_1,\dots,v_q] \rangle _{V}$	
	
	\textbf{Ensure:} Z-representation $\mathcal{P} = \langle c,G,\mathcal{E} \rangle_Z$ of the polytope
	\begin{algorithmic}[1]
		\State $\mathcal{K} = \emptyset$
		\For{$i \gets 1$ to $q$} \label{line:beginForLoop1}
			\State $\mathcal{K} \gets \big( \mathcal{K},~ \langle v_i,[~],\emptyset \rangle_Z \big)$
		\EndFor \label{line:endForLoop1}
		\While{$|\mathcal{K}| > 1$}	\label{line:beginOuterWhile}		
			\State $\widehat{\mathcal{K}} \gets \emptyset$				
			\While{$|\mathcal{K}| \geq 2$} \label{line:beginInnerWhile}
				\State $\widehat{\mathcal{K}} \gets \big( \widehat{\mathcal{K}} ~, conv(\mathcal{K}_{(1)}, \mathcal{K}_{(2)}) \big)$
				\If{$|\mathcal{K}| > 2$}				
					\State $\mathcal{K} \gets \big(\mathcal{K}_{(3)}, \dots, \mathcal{K}_{(|\mathcal{K}|)}\big)$
				\Else
					\State $\mathcal{K} \gets \emptyset$
				\EndIf
			\EndWhile \label{line:endInnerWhile}
			\If{$|\mathcal{K}| == 1$}
				\State $\widehat{\mathcal{K}} \gets \big( \widehat{\mathcal{K}}, ~ \mathcal{K}_{(1)} \big)$
			\EndIf
			\State $\mathcal{K} \gets \widehat{\mathcal{K}}$
		\EndWhile \label{line:endOuterWhile}	
		\State $\langle c,G,\mathcal{E} \rangle_Z \gets \mathcal{K}_{(1)}$
	\end{algorithmic}
\end{algorithm}

\noindent The algorithm is structured as follows: First, all vertices of the polytope are converted to Z-representation during the for-loop in lines \ref{line:beginForLoop1}-\ref{line:endForLoop1}, and the result is stored in the tuple $\mathcal{K}$. The remainder of the algorithm can then be viewed as the exploration of a binary tree as it is shown in Fig. \ref{fig:AlgDescription}, where the nodes of the tree are polytopes in Z-representation. Each iteration of the outer while-loop in lines \ref{line:beginOuterWhile}-\ref{line:endOuterWhile} of Alg. \ref{alg:Vrep2Prep} represents the exploration of one level of the tree. For each of these levels, the inner while-loop in lines \ref{line:beginInnerWhile}-\ref{line:endInnerWhile} visits all nodes at one level and computes the convex hull of two nodes to form one node of the next higher level of the tree. If the root node of the binary tree is reached, all polytope vertices have been united by the convex hull, and the root element is the desired Z-representation of the polytope $\mathcal{P}$. We demonstrate Alg. \ref{alg:Vrep2Prep} with a short example:

\begin{example}
	The conversion of the polytope 
	\begin{equation*}
		\mathcal{P} = \left \langle \left [ \begin{bmatrix} 0 \\ 5 \end{bmatrix},\begin{bmatrix} 3 \\ 6 \end{bmatrix}, \begin{bmatrix} 4 \\ 5 \end{bmatrix}, \begin{bmatrix} 5 \\ 1 \end{bmatrix}, \begin{bmatrix} 2 \\ 0 \end{bmatrix}, \begin{bmatrix} 0 \\ 2 \end{bmatrix} \right ] \right \rangle_V,
	\end{equation*}
	from V-representation to Z-representation with Alg. \ref{alg:Vrep2Prep} is visualized in Fig. \ref{fig:ExampleAlgorithm}. The algorithm terminates after $3$ iterations.
	\label{ex:algorithm}
\end{example}

\begin{figure}
	\centering
  	\includegraphics[width = 0.9 \textwidth]{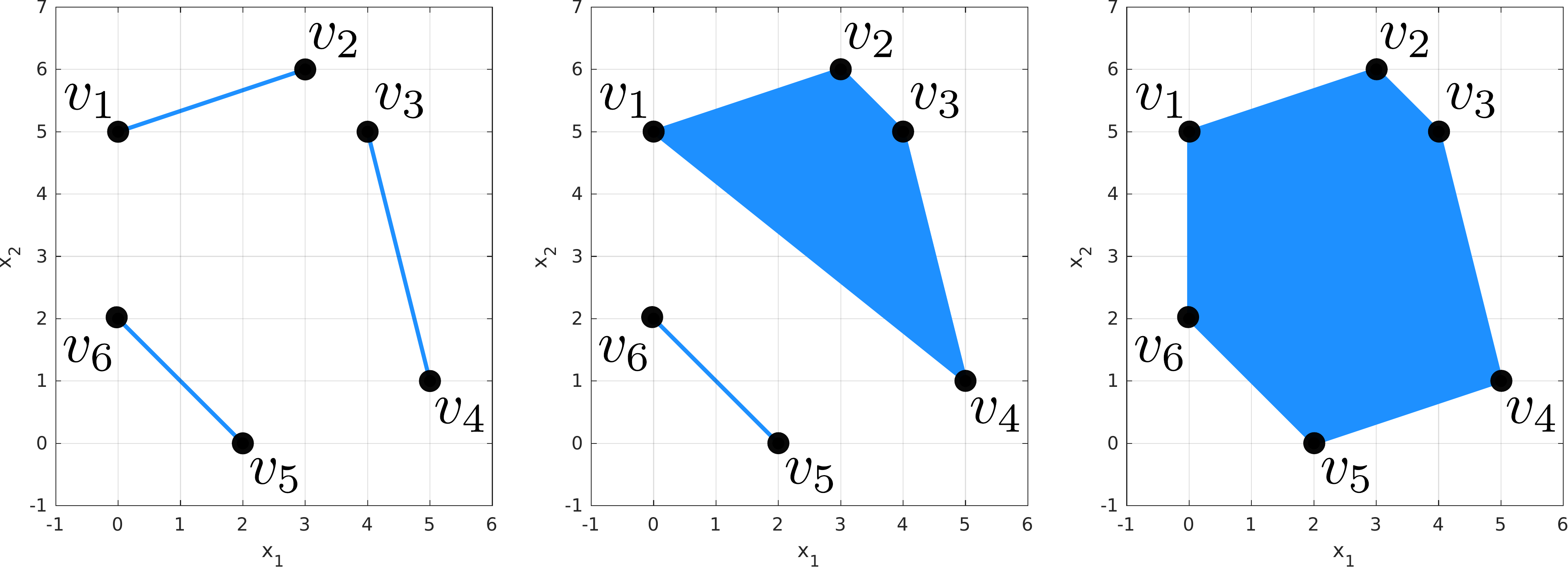}
	\captionof{figure}{First (left), second (middle), and third (right) iteration of Alg. \ref{alg:Vrep2Prep} applied to the polytope $\mathcal{P}$ from Example \ref{ex:algorithm}.}
	\label{fig:ExampleAlgorithm}
\end{figure}

The Z-representation of a polytope is not unique. If Alg. \ref{alg:Vrep2Prep} is used for the conversion of a polytope given in V-representation, then the resulting Z-representation depends on the order of the vertices in the matrix $[v_1,\dots,v_q]$, since this order defines which vertices are combined by the convex hull operation. Therefore, it might be meaningful to sort the matrix $[v_1,\dots,v_q]$ before applying Alg. \ref{alg:Vrep2Prep} in order to obtain a Z-representation with desirable properties. For example, to minimize the length of the vectors in the generator matrix of the Z-representation, the vertices have to be sorted so that vertices located close to each other are combined first. The computational complexity of the algorithm can be derived as follows:

\begin{proposition}
	The computational complexity of the conversion from V-representation to Z-representation using Alg. \ref{alg:Vrep2Prep} is $\mathcal{O}(q^2 \log_2(q) + nq)$ with respect to $q$, where $n$ is the dimension and $q$ is the number of polytope vertices.
\end{proposition}

\begin{proof} The for-loop in lines \ref{line:beginForLoop1}-\ref{line:endForLoop1} of Alg. \ref{alg:Vrep2Prep} can be ignored since it only involves initializations. Let us first consider the case where $q = 2^k$, $k \in \mathbb{N}$: each iteration of the outer while-loop in lines \ref{line:beginOuterWhile}-\ref{line:endOuterWhile} of Alg. \ref{alg:Vrep2Prep} corresponds to one level of a perfect binary tree with depth $k = \log_2(q)$ (see Fig. \ref{fig:AlgDescription}). Each node at level $i = 0 \dots k$ is a polytope in Z-representation $\mathcal{P}^{(i)} = \langle c^{(i)}, G^{(i)}, \mathcal{E}^{(i)} \rangle_Z$ with $G^{(i)} \in \mathbb{R}^{n \times h^{(i)}}$, and $\mu^{(i)}$ denoting the number of entries in the list $\mathcal{E}^{(i)}$ (see \eqref{eq:numListElements}). From \eqref{eq:changeProp} we can derive the number of factors $p^{(i)}$, the number of generators $h^{(i)}$, and the number of tuple entries $\mu^{(i)}$ of a node at level $i$ of the binary tree; the values for a perfect binary tree on level $i$ are:

\begin{equation}
	\begin{split}
		& p^{(i)} = 2 p^{(i-1)} + 1 = 2^i p^{(0)} + \sum_{j=0}^{i-1} 2^j \\
		& h^{(i)} = 4 h^{(i-1)} + 1 = 4^i h^{(0)} + \sum_{j=0}^{i-1} 4^j \\
		& \mu^{(i)} = 4 \mu^{(i-1)} + 2 h^{(i-1)} + 1 = 4 \mu^{(0)} + \sum_{j=0}^{i-1} 4^j \left( 1 + 2 h^{(i-1-j)} \right).
	\end{split}
	\label{eq:paramPrepIni}
\end{equation}
The nodes at the bottom level of the binary tree are the polytope vertices $v_l$ represented in Z-representation $v_l = \langle v_l, [~], \emptyset \rangle_Z$, so that $p^{(0)} = 0$, $h^{(0)} = 0$, and $\mu^{(0)} = 0$. Inserting these values into \eqref{eq:paramPrepIni} and using the finite sum of the geometric series
\begin{equation*}
	\sum_{j=0}^z r^j = \frac{1-r^{z+1}}{1-r}, ~~ r \in \mathbb{R},
\end{equation*}
one obtains 
\begin{equation}
	\begin{split}
		& p^{(i)} = \sum_{j=0}^{i-1} 2^j \overset{\mathrm{geom.}}{\overset{\mathrm{series}}{=}} 2^i - 1 \\
		& h^{(i)} = \sum_{j=0}^{i-1} 4^j \overset{\mathrm{geom.}}{\overset{\mathrm{series}}{=}} \frac{4^i - 1}{3} \\
		& \mu^{(i)} = \sum_{j=0}^{i-1} 4^j \left(1 + \frac{2}{3}(4^{i-1-j} - 1)\right) = \frac{1}{3} \sum_{j=0}^{i-1} 4^j + \frac{2}{3} \sum_{j=0}^{i-1} 4^{i-1} \overset{\mathrm{geom.}}{\overset{\mathrm{series}}{=}} 4^i \left( \frac{1}{6} i + \frac{1}{9} \right) - \frac{1}{9}.
	\end{split}
	\label{eq:paramPrep}
\end{equation}
Each level $i$ of the tree contains $2^{k-i}$ nodes, where $k$ is the depth of the tree. Consequently, at each level $2^{k-i}$ convex hull operations have to be computed, where each operation involves $2n + \mu^{(i-1)}$ additions (see Prop. \ref{prop:convexHullPrep}). The number of necessary basic operations $O$ required for the conversion with Alg. \ref{alg:Vrep2Prep} is therefore 
\begin{equation}
	O = \sum_{i=1}^{k} 2^{k-i} \left(2n + \mu^{(i-1)} \right).
	\label{eq:nrOfOperationsBasic}
\end{equation}
In the general case the binary tree explored by Alg. \ref{alg:Vrep2Prep} is not a perfect binary tree. However, the number of operations required for the general case is obviously smaller than the number of operations required for the exploration of a perfect binary tree with depth $\lceil \log_2(q) \rceil$. Since it holds that $\lceil \log_2(q) \rceil = \log_2(q) + a$, $a \in [0,1]$, inserting $k=\lceil \log_2(q) \rceil = \log_2(q) + a$ for the tree depth in \eqref{eq:nrOfOperationsBasic} yields 
\begin{equation*}
    \begin{split}
	& ~ O = \sum_{i=1}^{\log_2(q) + a} 2^{\log_2(q)+a-i} \left(2n + \mu^{(i-1)} \right) = 2^a \sum_{i=1}^{\log_2(q) + a}  \frac{q}{2^i} \left(2n + \mu^{(i-1)} \right) \\
	& ~ \\
	&  \overset{\overset{\mu^{(i-1)}}{\mathrm{from}~ \eqref{eq:paramPrep}}}{=} 2^a \left(2n-\frac{1}{9}\right)q \underbrace{\sum_{i=1}^{\log_2(q)+a} \frac{1}{2^i}}_{\overset{\mathrm{geom.}}{\overset{\mathrm{series}}{=}} 1-\frac{1}{2^a q}} - \frac{2^a}{72}q \underbrace{\sum_{i=1}^{\log_2(q)+a} 2^i}_{\overset{\mathrm{geom.}}{\overset{\mathrm{series}}{=}} 2 (2^a q-1)} + \frac{2^a}{24} q \underbrace{\sum_{i=1}^{\log_2(q)+a} i \cdot 2^i}_{\overset{\mathrm{(see}~\text{\cite[Chapter 1.2.2.3]{Jeffrey2008}} \mathrm{)}}{= 2(1 + 2^a q(a-1) + 2^a \log_2(q)q)}} \\
	& ~~~~ \\
	& ~~~~ = \frac{4^a}{12} q^2 \log_2(q) + \frac{4^a}{36} (3a-4) q^2 + 2^{a+1}n q - 2n + \frac{1}{9}.
	\end{split}
	\label{eq:nrOfOperations}
\end{equation*}  
The complexity of Alg. \ref{alg:Vrep2Prep} is therefore $\mathcal{O}(q^2 \log_2(q) + nq)$ with respect to $q$. \hfill $\square$
\end{proof}

\begin{figure}
	\centering
  	\includegraphics[width = 0.6 \textwidth]{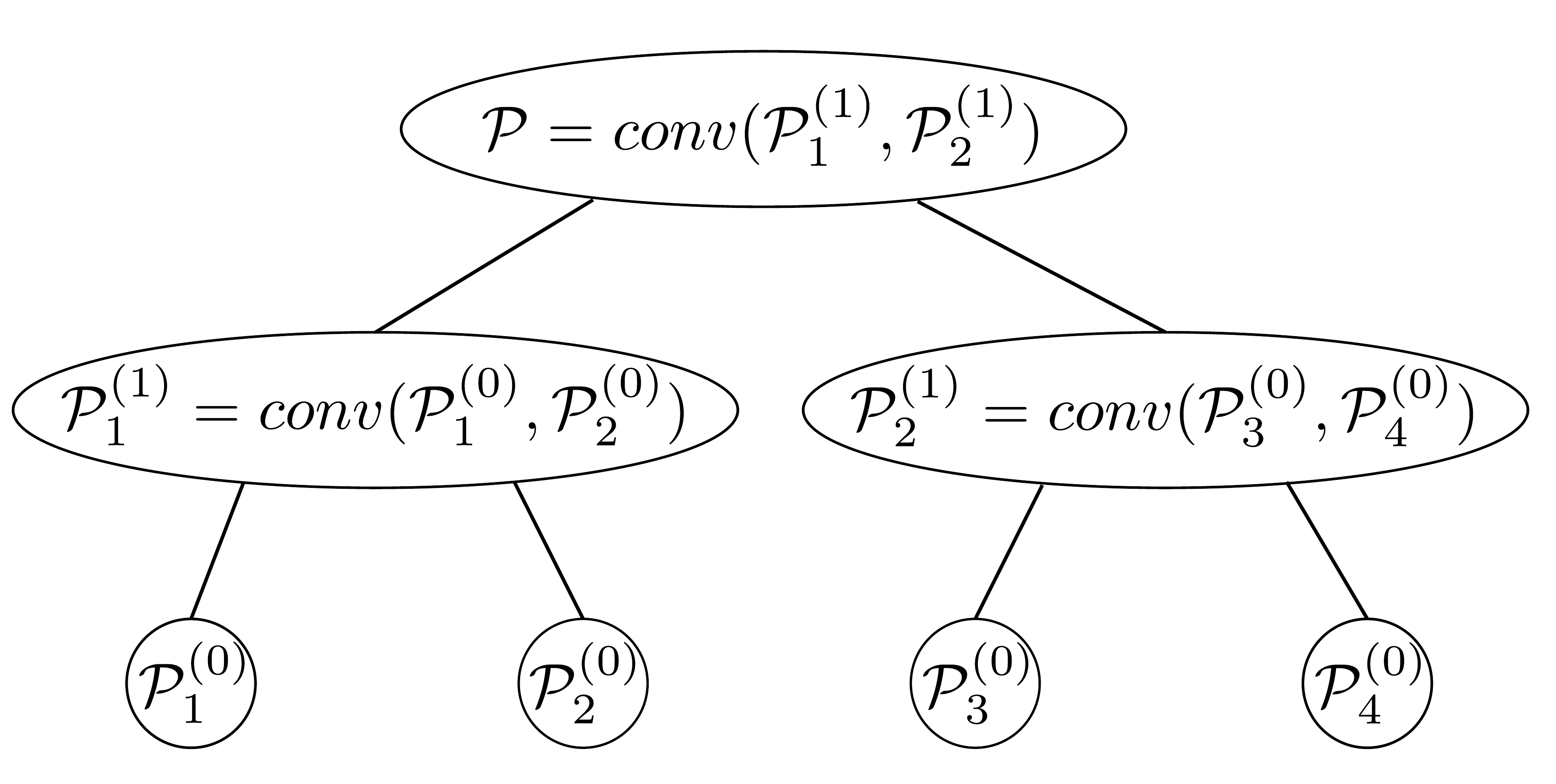}
	\captionof{figure}{Example of a perfect binary tree as explored by Alg. \ref{alg:Vrep2Prep}.}
	\label{fig:AlgDescription}
\end{figure}

\subsection{Conversion to V-Representation}

In this subsection we present an algorithm that converts a polytope in Z-representation to one in V-representation. Our algorithm is based on the following proposition:
\begin{proposition}
	Given a polytope $\mathcal{P} = \langle c, G, \mathcal{E} \rangle_Z$ in Z-representation, the polytope vertices are a subset of the finite set $\mathcal{K}$ defined as  
	\begin{equation*}
		\mathcal{K} = \left\{ c + \sum _{i=1}^h \left( \prod _{k=1}^{m_i} \widehat{\alpha}_{\mathcal{E}_{(i,k)}} \right) G_{(\cdot,i)} ~\bigg| ~ \widehat{\alpha} = [ \widehat{\alpha}_1, \dots, \widehat{\alpha}_p ]^T \in \operator{vertices}(\mathcal{I}) \right\} ,
	\end{equation*}
	where the operation \operator{vertices} returns the $2^p$ vertices of the hypercube $\mathcal{I} = [-\mathbf{1},\mathbf{1}] \subset \mathbb{R}^p$.
	\label{prop:verticesHypercube}
\end{proposition}

\begin{proof}
	We have to show that each vertex $v^{(j)}$ of the polytope $\mathcal{P}$ corresponds to one vertex $\widehat{\alpha}^{(j)}$ of the hypercube $\mathcal{I}$:
	\begin{equation}
		v^{(j)} = c + \sum _{i=1}^h \left( \prod _{k=1}^{m_i} \widehat{\alpha}_{\mathcal{E}_{(i,k)}}^{(j)} \right) G_{(\cdot,i)}.
		\label{eq:correspVert}
	\end{equation}	
As shown in \cite[Chapter 7.2(d)]{Padberg2013}, for each vertex $v^{(j)}$ of a polytope $\mathcal{P}$ there exists a vector $d_j \in \mathbb{R}^n$ such that 
	\begin{equation}
		v^{(j)} = \argmax_{s \in \mathcal{P}} ~~ d_j^T s.
		\label{eq:polyVertProj}
	\end{equation} 
	By using \eqref{eq:correspVert}, \eqref{eq:polyVertProj} can be equivalently formulated as 
	\begin{equation}
		\begin{split}
		& v^{(j)} = c + \sum _{i=1}^h \left( \prod _{k=1}^{m_i} \alpha_{\mathcal{E}_{(i,k)}}^* \right) G_{(\cdot,i)} \\
		& \mathrm{with} ~~ [\alpha^*_1,\dots,\alpha^*_p]^T = \argmax_{[\alpha_1, \dots, \alpha_p]^T \in \mathcal{I}} \underbrace{d_j^T \left( c + \sum _{i=1}^h \left( \prod _{k=1}^{m_i} \alpha_{\mathcal{E}_{(i,k)}}\right) G_{(\cdot,i)} \right)}_{f(\alpha_1,\dots,\alpha_p)}.
		\end{split}
		\label{eq:polyFunc}
	\end{equation}
We therefore have to show that the point $\alpha^*=[\alpha^*_1,\dots,\alpha^*_p]^T$ where the function $f(\cdot)$ from \eqref{eq:polyFunc} reaches its extremum within the domain $[\alpha_1, \dots, \alpha_p]^T \in \mathcal{I}$ is identical to a vertex $\widehat{\alpha}^{(j)}$ of the hypercube $\mathcal{I}$. Since the function $f(\cdot)$ from \eqref{eq:polyFunc} does not contain polynomial exponents greater than $1$ (see \eqref{eq:pRep}), it holds that the partial derivative of $f(\cdot)$ with respect to variable $\alpha_i$ does not depend on $\alpha_i$: 

\begin{equation}
	\forall ~ i \in \{1,\dots,p \}: ~~ \frac{\partial f(\alpha_1,\dots,\alpha_p)}{\partial \alpha_i} = g_i(\alpha_1,\dots,\alpha_{i-1},\alpha_{i+1},\dots,\alpha_p).
	\label{eq:partialDer}
\end{equation}
According to \eqref{eq:partialDer} the function $f(\cdot)$ therefore reaches its extremum on the domain $\alpha_i \in [-1,1]$ at either $\alpha_i^* = 1$ or $\alpha_i^* = -1$. Since this holds for all $\alpha_i$, $i = 1,\dots,p$, function $f(\cdot)$ reaches its extremum within the domain $[\alpha_1, \dots, \alpha_p]^T \in \mathcal{I}$ at the point $\alpha^*=[\alpha^*_1,\dots,\alpha^*_p]^T$ with $\alpha_j^* \in \{-1,1\},~ j = 1,\dots,p$, which is a vertex of the hypercube $\mathcal{I}$. \hfill $\square$
\end{proof}

\begin{algorithm}[H]
	\caption{Conversion from Z-representation to V-representation} \label{alg:Prep2Vrep}
	\textbf{Require:} Bounded polytope in Z-representation $\mathcal{P} = \langle c, G, \mathcal{E} \rangle_Z$
	
	\textbf{Ensure:} V-representation $\mathcal{P} = \langle [v_1,\dots,v_q] \rangle_V$ of the polytope
	\begin{algorithmic}[1]
		\State $\mathcal{I} \gets [-\mathbf{1},\mathbf{1}] \subset \mathbb{R}^p$
		\State $\big\{ \widehat{\alpha}^{(1)}, \dots, \widehat{\alpha}^{(2^p)}\big\} \gets \operator{vertices}(\mathcal{I})$ \label{line:vertices}
		\State $\mathcal{K} \gets \emptyset$
		\For{$j \gets 1$ to $2^p$} \label{line:beginForLoop}
			\State $v \gets  c + \sum _{i=1}^h \left( \prod _{k=1}^{m_i} \widehat{\alpha}^{(j)}_{\mathcal{E}_{(i,k)}} \right) G_{(\cdot,i)}$ \label{line:potVertices}
			\If{$v \notin \mathcal{K}$} \label{line:checkContain}
				\State $\mathcal{K} \gets \mathcal{K} \cup v$
			\EndIf
		\EndFor \label{line:endForLoop}
		\State $[v_1,\dots,v_q] \gets \operator{convexHull}(\mathcal{K})$ \label{line:convHull}
		\State $\mathcal{P} \gets \langle [v_1,\dots,v_q] \rangle _{V}$
	\end{algorithmic}
\end{algorithm}

Alg. \ref{alg:Prep2Vrep} shows, based on Prop. \ref{prop:verticesHypercube}, the conversion from Z-representation to V-representation. The operation \operator{vertices} in lines \ref{line:vertices} of Alg. \ref{alg:Prep2Vrep} returns the vertices of the hypercube $\mathcal{I}$. In the for-loop in line \ref{line:beginForLoop}-\ref{line:endForLoop} of Alg. \ref{alg:Prep2Vrep} we compute the potential polytope vertices for each of these hypercube vertices according to Prop. \ref{prop:verticesHypercube}. We check if the potential vertex $v$ is already part of the set $\mathcal{K}$ in line \ref{line:checkContain} since this decreases the average runtime of the algorithm. The points stored in $\mathcal{K}$ define a redundant V-representation of the polytope $\mathcal{P}$. Redundant points are removed by computation of the convex hull in line \ref{line:convHull} of Alg. \ref{alg:Prep2Vrep}, where the operation \operator{convexHull} returns the vertices of the convex hull. For Z-representations that define non-convex sets (see e.g., Example \ref{ex:pRep2}), Alg. \ref{alg:Prep2Vrep} returns the convex hull of the set. 

Since it is required for the derivation of the computation complexity, we first derive a formula for the maximum number of generators and the maximum number of tuple entries for a regular Z-representation with a fixed number of factors:

\begin{proposition}
	Given a regular Z-representation $\mathcal{P} = \langle c,G,\mathcal{E} \rangle_Z$ with $p$ factors, it holds that 
	\begin{equation}
	\begin{split}
		& h(p) \leq 2^p - 1 \\
		& \mu(p) \leq p ~ 2^{p-1}
	\end{split}
		\label{eq:propOfP}
	\end{equation}
	are upper bounds for the number of generators $h(p)$ and the number of entries $\mu(p)$ in the tuple $\mathcal{E}$.
	\label{prop:numGen}
\end{proposition}

\begin{proof}
	For a regular Z-representation (see \eqref{eq:regular}), the maximum number of generators is equal to the number of different monomial variable parts $\alpha_{\mathcal{E}_{(i,1)}} \cdot \ldots \cdot \alpha_{\mathcal{E}_{(i,m_i)}}$ that can be constructed with $p$ factors $\alpha_{\mathcal{E}_{(i,k)}}$. Given a fixed $m_i$, there exist ${p \choose m_i}$ different monomial variable parts since there are ${p \choose m_i}$ possible combinations to choose $m_i$ from the $p$ factors $\alpha _{\mathcal{E}_{(i,k)}}$ without order. Summation over all $m_i \in \{ 1,\dots,p \}$ therefore yields
	\begin{equation*}
		h(p) \leq \sum_{m_i=1}^p {p \choose m_i} = \left( \sum_{m_i=0}^p {p \choose i } \right)-1 \overset{\text{\cite[Chapter~8.6~(7.)]{Rade2013}}}{=} 2^p - 1.
	\end{equation*}
	For each monomial, $m_i$ entries have to be stored in the tuple $\mathcal{E}$. Consequently,
	\begin{equation*}
		\mu(p) \leq \sum_{m_i=1}^p m_i {p \choose m_i} = \sum_{m_i=0}^p m_i {p \choose m_i } \overset{\text{\cite[Chapter~8.6~(8.)]{Rade2013}}}{=} p ~ 2^{p-1}
	\end{equation*}
	is an upper bound for the number of entries $\mu$ in the tuple $\mathcal{E}$. \hfill $\square$
\end{proof}

Next, we derive the computational complexity of Alg. \ref{alg:Prep2Vrep}:

\begin{proposition}
	The computational complexity of the conversion of a polytope $\mathcal{P} \subset \mathbb{R}^n$ from a regular Z-representation to V-representation with Alg. \ref{alg:Prep2Vrep} is $\mathcal{O}((2^p)^{\lfloor n/2 \rfloor + 1} + 4^p (p + n))$ with respect to $p$, where $n$ is the dimension and $p$ is the number of factors of the Z-representation.
\end{proposition}

\begin{proof}
	 Upper bounds for the number of generators $h(p)$ and the number of entries $\mu(p)$ in the tuple $\mathcal{E}$ for a regular Z-representation with $p$ factors are given by Prop. \ref{prop:numGen}. The hypercube $\mathcal{I} = [-\mathbf{1},\mathbf{1}] \subset \mathbb{R}^{p}$ has $2^{p}$ vertices. Therefore, the for-loop in lines \ref{line:beginForLoop}-\ref{line:endForLoop} of Alg. \ref{alg:Prep2Vrep} consists of $2^{p}$ iterations. In each iteration, the potential vertex $v$ has to be calculated according to Prop. \ref{prop:verticesHypercube}. This requires $n h(p)$ additions and at most $\mu(p) -1 + nh(p)$ multiplications, resulting for all iterations in 
\begin{equation}
	O_1 = 2^{p} \left( 2nh(p) + \mu(p) - 1 \right) \overset{\eqref{eq:propOfP}}{\leq} 4^p \left(\frac{1}{2}p + 2n \right) - 2n~ 2^p - 1
	\label{eq:complexity1}
\end{equation}
basic operations. Furthermore, checking if the point is already contained in the tuple $\mathcal{K}$ storing the potential vertices in line \ref{line:checkContain} of Alg. \ref{alg:Prep2Vrep} requires, in the worst case,
\begin{equation}
	O_2 = \sum_{i=1}^{2^{p(q)}} n (i-1) \overset{\text{\cite[Chapter~1.2.2.1]{Jeffrey2008}}}{=} \frac{n}{2} \left( 4^{p} - 2^{p} \right)
	\label{eq:complexity2}
\end{equation}
comparisons of scalar numbers. The complexity of the computations in the for-loop in lines \ref{line:beginForLoop}-\ref{line:endForLoop} of Alg. \ref{alg:Prep2Vrep} is therefore $\mathcal{O}(O_1 + O_2) = \mathcal{O}(4^p (p + n))$ with respect to $p$. The Beneath-Beyond algorithm \cite{Kallay1981} for the computation of the convex hull of a $n$-dimensional point cloud with $w$ points has complexity $\mathcal{O}(w^{\lfloor n/2 \rfloor + 1})$ \cite[Theorem 3.16]{Preparata2012}. For Alg. \ref{alg:Prep2Vrep}, the point cloud stored in the set $\mathcal{K}$ consists of $w = 2^p$ points in the worst case, resulting in complexity $\mathcal{O}((2^p)^{\lfloor n/2 \rfloor + 1})$ with respect to $p$ for the computation of the convex hull in line \ref{line:convHull} of Alg. \ref{alg:Prep2Vrep}. Combining the complexity from the for-loop and from the convex hull computation results in an overall complexity of
	\begin{equation}
		\mathcal{O}(O_1 + O_2) + \mathcal{O}((2^p)^{\lfloor n/2 \rfloor + 1}) = \mathcal{O}(4^p (p + n)) + \mathcal{O}((2^p)^{\lfloor n/2 \rfloor + 1})
		\label{eq:overallComplexity}
	\end{equation} 
	with respect to $p$. \hfill $\square$
\end{proof}

For general Z-representations, it is not possible to specify a relation between the number of polytope vertices and the number of factors. However, under the assumption that the Z-representation is obtained by conversion from V-representation with Alg. \ref{alg:Vrep2Prep}, the number of factors can be expressed as a function of the number of vertices, which enables us to derive the computation complexity with respect to the number of polytope vertices:

\begin{proposition}
	The computational complexity of the conversion of a polytope $\mathcal{P}$ from Z-representation to V-representation with Alg. \ref{alg:Prep2Vrep}, where $\mathcal{P}$ is computed by conversion from V-representation with Alg. \ref{alg:Vrep2Prep}, is $\mathcal{O}(16^q n + (0.5 \cdot 4^q)^{\lfloor n/2 \rfloor + 1})$ with respect to $q$, where $n$ is the dimension and $q$ is the number of polytope vertices.
\end{proposition}

\begin{proof} We first express the number of factors $p$, the number of generators $h$, and the number of list elements $\mu$ (see \eqref{eq:numListElements}) of a polytope in Z-representation as functions of the number of polytope vertices $q$. As stated in the proposition, we assume that the Z-representation of the polytope is obtained by conversion from V-representation using Alg. \ref{alg:Vrep2Prep}. As shown in Sec. \ref{subsec:Vrep2Prep}, we obtain over-approximations for $p$, $h$, and $\mu$ if we view Alg. \ref{alg:Vrep2Prep} as the exploration of a perfect binary tree with depth $k = \lceil \log_2(q) \rceil$. Since it holds that $\lceil \log_2(q) \rceil = \log_2(q) + a$, $a \in [0,1]$, we insert the tree depth $\log_2(q) + a$ into \eqref{eq:paramPrep} to obtain
\begin{equation}
	\begin{split}
		& p(q) = p^{(\log_2(q)+a)} = 2^a q - 1 \\
		& h(q)= h^{(\log_2(q)+a)} = \frac{4^a q^2-1}{3} \\
		& \mu(q) = \mu^{(\log_2(q)+a)} = 4^a q^2  \left( \frac{1}{6} \log_2(q) + \frac{1}{6} a + \frac{1}{9} \right) - \frac{1}{9}.
	\end{split}
	\label{eq:propVert}
\end{equation} 
The overall complexity for the conversion with Alg. \ref{alg:Prep2Vrep} is given by \eqref{eq:overallComplexity}. According to  \eqref{eq:complexity1} and \eqref{eq:complexity2}, it holds that
\begin{equation}
	O_1 = 2^{p(q)} \left( 2nh(q) + \mu(q) - 1 \right) \overset{\eqref{eq:propVert}}{=} (4^a)^{q} \left( 4^a q^2 \left( \frac{1}{3} n + \frac{1}{12} a + \frac{1}{18} \right) + \frac{1}{12} 4^a q^2 \log_2(q) - \frac{1}{3} n - \frac{5}{9}  \right)
	\label{eq:newComplexity1}
\end{equation}
and
\begin{equation}
	O_2 = \frac{n}{2} \left( 4^{p(q)} - 2^{p(q)} \right) \overset{\eqref{eq:propVert}}{=} \frac{n}{4} (16^a)^q - \frac{n}{2} (4^a)^q.
	\label{eq:newComplexity2}
\end{equation}
Inserting \eqref{eq:newComplexity1} and \eqref{eq:newComplexity2} into \eqref{eq:overallComplexity} yields
\begin{equation*}
	\mathcal{O}(O_1 + O_2) + \mathcal{O}((2^{p(q)})^{\lfloor n/2 \rfloor + 1}) \overset{\eqref{eq:propVert},\eqref{eq:newComplexity1},\eqref{eq:newComplexity2}}{=} \mathcal{O}((16^a)^q n) + \mathcal{O}((0.5 (4^a)^q)^{\lfloor n/2 \rfloor + 1})
\end{equation*}
with respect to $q$, which is $\mathcal{O}(16^q n + (0.5 \cdot 4^q)^{\lfloor n/2 \rfloor + 1})$. \hfill $\square$
\end{proof}


\section{Representation Complexity}

In this section, we compare the representation complexity in V-representation, H-representation, and Z-representation for two special classes of polytopes. Given a polytope $\mathcal{P}$, we denote by $N_V(\mathcal{P})$, $N_H(\mathcal{P})$ and $N_Z(\mathcal{P})$ the number of values required for V-, H-, and Z-representation, respectively. Furthermore, we denote the number of $i$-dimensional polytope faces by $F_i(\mathcal{P})$. We first derive the representation complexity for V-, H-, and Z-representation:

\begin{proposition}
The representation complexity of a polytope $\mathcal{P}$ with $q$ vertices in V-representation is 
\begin{equation*}
	N_V(\mathcal{P}) = nq.
\end{equation*}
\label{prop:repCompVrep}
\end{proposition}
\begin{proof}
 For each of the $q$ vertices, a vector with $n$ entries has to be stored. \hfill $\square$
\end{proof}

\begin{proposition}
The representation complexity of an $n$-dimensional polytope $\mathcal{P}$ in H-representation is 
\begin{equation*}
	N_H(\mathcal{P}) = (n+1) F_{n-1}(\mathcal{P}).
\end{equation*}
\label{prop:repCompHrep}
\end{proposition}
\begin{proof}
 Each of the $F_{n-1}(\mathcal{P})$ $(n-1)$-dimensional polytope facets corresponds to one inequality constraint consisting of $n+1$ values. \hfill $\square$
\end{proof}

\begin{proposition}
The representation complexity of a polytope $\mathcal{P} = \langle c,G,\mathcal{E}\rangle_Z$ in Z-representation is 
\begin{equation*}
	N_Z(\mathcal{P}) = n(h+1) + \mu,
\end{equation*}
where $h$ is the number of generators and $\mu$ is the number of entries in the tuple $\mathcal{E}$. 
\label{prop:repCompPrep}
\end{proposition}
\begin{proof}
 The center vector $c \in \mathbb{R}^n$ consists of $n$ values, the generator matrix $G \in \mathbb{R}^{n \times h}$ of $nh$ values, and the tuple $\mathcal{E}$ of $\mu$ values (see \eqref{eq:numListElements}). \hfill $\square$
\end{proof}

\subsection{Convex Hull of a Zonotope and a Point}
\label{subsec:zonoPoint}

We first consider the special case of a polytope $\mathcal{C} = conv(\mathcal{Z},d)$ that can be described by the convex hull of a zonotope $\mathcal{Z} \subset \mathbb{R}^n$ and a single point $d \in \mathbb{R}^n$. A zonotope is a special type of a polytope or a polynomial zonotope \cite{Ziegler1995}. Using the shorthand for the Z-representation, a zonotope is $\mathcal{Z} = \langle c,G,(1, \dots, m) \rangle_Z$, where $c \in \mathbb{R}^n$ is the center vector, $G \in \mathbb{R}^{n \times m}$ is the generator matrix, and $m$ is the number of zonotope generators. Next, we derive the representation complexity of the polytope $\mathcal{C}$ in V-, H-, and Z-representation.

\myparagraph{V-Representation}

\noindent The number of vertices $q_n(m)$ of an $n$-dimensional zonotope with $m$ generators is \cite[Prop. 2.1.2]{Gritzmann1993}
\begin{equation}
    q_n(m) = 2 \sum_{i=0}^{\min(n,m)-1} {m-1 \choose i }.
	\label{eq:zonotopeVertices}
\end{equation}
The polytope $\mathcal{C}$ defined by the convex hull of a zonotope and a point has, in the worst case, $q_n(m)+1$ vertices. According to Prop. \ref{prop:repCompVrep},
\begin{equation*}
	N_V(\mathcal{C}) \leq n (q_n(m) + 1) \overset{\eqref{eq:zonotopeVertices}}{=} n \left( 1 + 2 \sum_{i=0}^{\min(n,m)-1} {m-1 \choose i }  \right) 
\end{equation*}
is an upper bound for the representation complexity.

\myparagraph{H-Representation}

\noindent An $n$-dimensional zonotope $\mathcal{Z}$ with $m$ generators has at most $2 {m \choose n-1 }$ facets \cite[Chapter 2.2.1]{Althoff2010a}. The maximum number of facets for the set $\mathcal{C} = conv(\mathcal{Z},d)$ is obtained in the case where $d \not\in \mathcal{Z}$, and all facets of $\mathcal{Z}$ except for one facet $\hat{\mathcal{F}}$ are facets of $\mathcal{C}$. It therefore holds that 
\begin{equation}
	F_{n-1}(\mathcal{C}) \leq \underbrace{F_{n-1}(\mathcal{Z}) - 1}_{F_A} + \underbrace{F_{n-1}(conv(\hat{\mathcal{F}},d)) - 1}_{F_B},
	\label{eq:numFacets}
\end{equation} 
where $F_A$ is the number of facets of $\mathcal{Z}$ that are facets of $\mathcal{C}$, and $F_B$ is the number of additional facets resulting from the convex hull of the facet $\hat{\mathcal{F}}$ with the point $d$. The number of facets for the convex hull of facet $\hat{\mathcal{F}}$ and point $d$ is identical to $F_{n-1}(conv(\hat{\mathcal{F}},d)) = F_{n-2}(\hat{\mathcal{F}}) + 1$, where $F_{n-2}(\hat{\mathcal{F}})$ is the number of $(n-2)$-dimensional faces of facet $\hat{\mathcal{F}}$. An upper bound for $F_{n-2}(\hat{\mathcal{F}})$ is $2 {m \choose n-2 }$, which corresponds to the case where $\hat{\mathcal{F}}$ is an $(n-1)$-dimensional zonotope with $m$ generators. Inserting $F_{n-1}(\mathcal{Z}) \leq 2 {m \choose n-1 }$ and $F_{n-1}(conv(\hat{\mathcal{F}},d)) \leq 2 {m \choose n-2 } + 1$ into \eqref{eq:numFacets} yields
\begin{equation}
	F_{n-1}(\mathcal{C}) \leq 2 {m \choose n-1 } - 1 + 2 {m \choose n-2 }.
	\label{eq:facesHrep}
\end{equation}
According to Prop. \ref{prop:repCompHrep}, an upper bound for the representation complexity is
\begin{equation*}
	N_H(\mathcal{C}) = (n+1) F_{n-1}(\mathcal{C}) \overset{\eqref{eq:facesHrep}}{\leq} (n+1) \left( 2 {m \choose n-1 } - 1 + 2 {m \choose n-2 } \right).
\end{equation*}

\myparagraph{Z-Representation}

\noindent For the zonotope $\mathcal{Z}= \langle c,G,(1, \dots, m) \rangle_Z$, we have $h_z=m$ and $\mu_z = m$. The point $d$ can be represented with the Z-representation $d = \langle d,[~],\emptyset \rangle_Z$ with $h_d = 0$ and $\mu_d = 0$. Computation of the convex hull $\mathcal{C} = conv(\mathcal{Z},d)$ therefore results according to \eqref{eq:changeProp} in a Z-representation with $h_c = 2h_z + 2h_d + 1 = 2m+1$ and $\mu_c = 2 \mu_z + 2 \mu_d + h_z + h_d + 1 = 3m+1$. The representation complexity for $\mathcal{C}$ according to Prop. \ref{prop:repCompPrep} is
\begin{equation*}
	N_Z(\mathcal{C}) = n(h_c+1) + \mu_c = 2n + 2mn + 3m + 1.
\end{equation*}

\subsection{Convex Hull of Two Zonotopes}
\label{subsec:convHullZono}

Next, we consider the case of a polytope $\mathcal{C} = conv(\mathcal{Z}_1,\mathcal{Z}_2)$ defined by the convex hull of a full-dimensional zonotope $\mathcal{Z}_1 = \langle c_1,G_1,(1, \dots, m_1) \rangle_Z$ with $m_1$ generators and a full-dimensional zonotope $\mathcal{Z}_2 = \langle c_2,G_2, \linebreak[1] (1, \dots, m_2) \rangle_Z$ with $m_2$ generators. For the Z-representation, the exact representation complexity can be computed. Since the number of facets and vertices of $\mathcal{C}$ depends on the shape of the two zonotopes, the exact representation complexity for the V-representation and the H-representation cannot be computed for the case of general zonotopes. We therefore consider the case where the zonotope with fewer generators encloses the second zonotope, which results in a minimum number of facets and vertices for $\mathcal{C}$, and therefore represents the best case for the V- and H-representation.

\myparagraph{V-Representation}

\noindent The number of vertices $q_n(m)$ of an $n$-dimensional zonotope with $m$ generators is given by \eqref{eq:zonotopeVertices}. A lower bound for the number of vertices of the polytope $\mathcal{C}$ is therefore $q_n(\min(m_1,m_2))$. According to Prop. \ref{prop:repCompVrep}, 
\begin{equation*}
	N_V(\mathcal{C}) \geq n ~ q_n(\min(m_1,m_2)) = 2n \sum_{i=0}^{\min(n,m_1,m_2)-1} {\min(m_1,m_2)-1 \choose i }
\end{equation*}
is a lower bound for the representation complexity.

\myparagraph{H-Representation}

\noindent An $n$-dimensional zonotope with $m$ generators has at most $2 {m \choose n-1 }$ facets \cite[Chapter 2.2.1]{Althoff2010a}. Therefore, $2 {\min(m_1,m_2) \choose n-1 }$ is a lower bound for the number of facets of the set $\mathcal{C}$. According to Prop. \ref{prop:repCompHrep},
\begin{equation*}
N_H(\mathcal{C}) \geq 2 {\min(m_1,m_2) \choose n-1 } (n+1)
\end{equation*}
is a lower bound for the representation complexity.

\myparagraph{Z-Representation}

\noindent For the two zonotopes $\mathcal{Z}_1 = \langle c_1,G_1,(1, \dots, m_1) \rangle_Z$ and  $\mathcal{Z}_2 = \langle c_2,G_2,(1, \dots, m_2) \rangle_Z$, we have $h_1=m_1$, $\mu_1 = m_1$, $h_2=m_2$, and $\mu_2 = m_2$. Computation of the convex hull $\mathcal{C} = conv(\mathcal{P}_1,\mathcal{P}_2)$ therefore results according to \eqref{eq:changeProp} in a Z-representation with $h_c = 2h_1 + 2h_2 + 1 = 2m_1 + 2m_2 + 1$ and $\mu_c = 2 \mu_1 + 2 \mu_2 + h_1 + h_2 + 1 = 3m_1 + 3m_2 + 1$. The representation complexity for $\mathcal{C}$ according to Prop. \ref{prop:repCompPrep} is 
\begin{equation*}
	N_Z(\mathcal{C}) = n_c(h_c + 1) + \mu_c = 2n (m_1 + m_2 + 1) + 3m_1 + 3m_2 + 1.
\end{equation*}

\subsection{Comparison}

\myparagraph{Convex Hull of a Zonotope and a Point}

\noindent The visualization of the representation complexity in Fig. \ref{fig:RepComplexity} shows that depending on the polytope the Z-representation can be significantly more compact than the V- and H-representation, especially for zonotopes with many generators. We further demonstrate this with a numerical example:

\begin{example}
	The polytope $\mathcal{C} = conv(\mathcal{I},d)$ corresponding to the convex hull of the point $d = [ 2 , \mathbf{0} ]^T \in \mathbb{R}^{20}$ and the hypercube $\mathcal{I} = [-\mathbf{1},\mathbf{1}] \subset \mathbb{R}^{20}$ has representation size $N_V(\mathcal{C}) = 10485770$, $N_H(\mathcal{C}) = 1617$, and $N_Z(\mathcal{C}) = 901$.
\end{example}

\begin{figure}
	\centering
  	\includegraphics[width = 0.5 \textwidth]{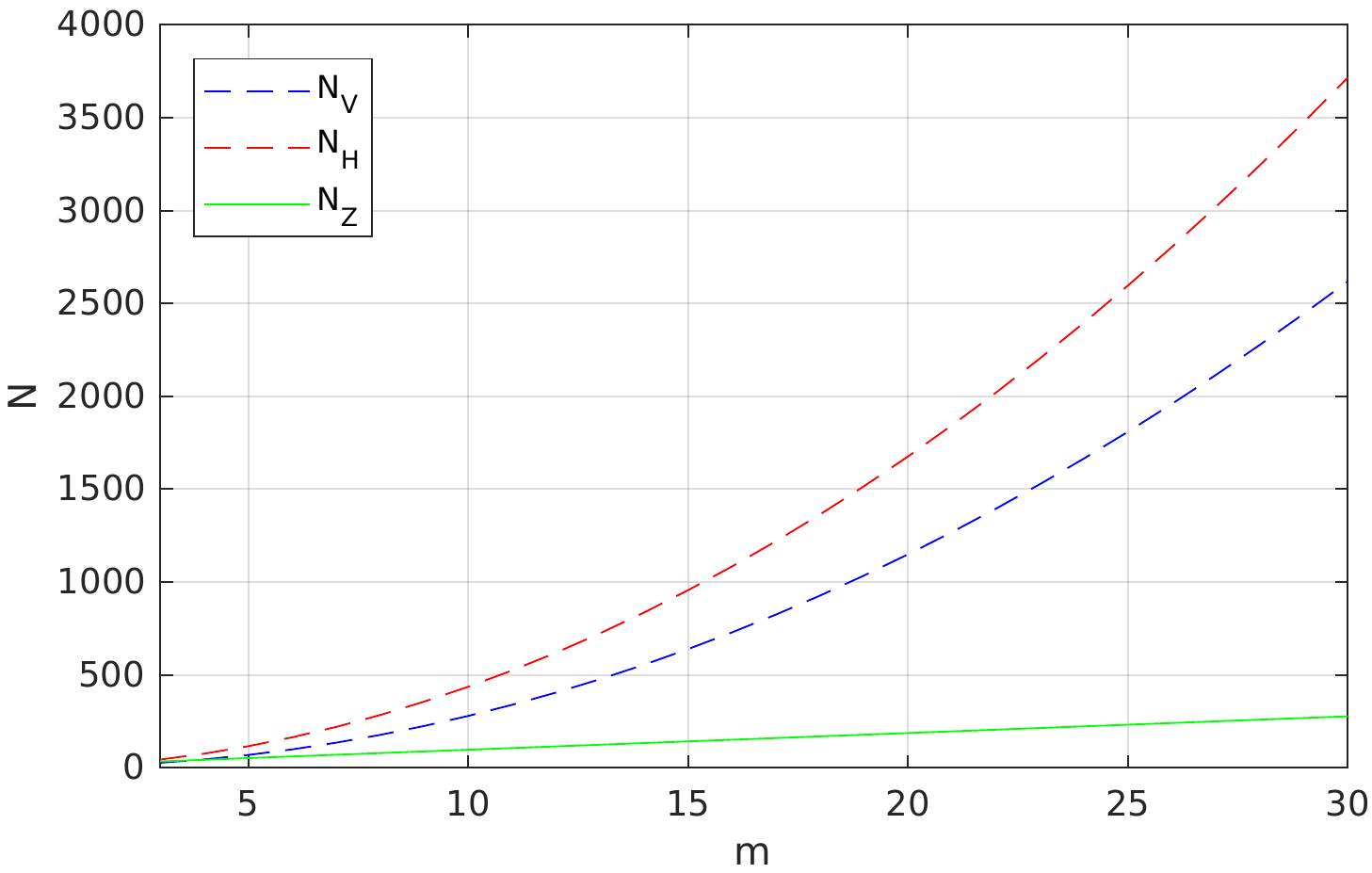}
	\captionof{figure}{Representation complexity of a 3-dimensional polytope defined by the convex hull of a zonotope with $m$ generators and a single point. Exact values are visualized by solid lines, and upper bounds are visualized by dashed lines.}
	\label{fig:RepComplexity}
\end{figure}

\myparagraph{Convex Hull of Two Zonotopes}

\noindent A visualization of the representation with the minimal representation complexity for different value combinations of $n$, $m_1$, and $m_2$ is shown in Fig. \ref{fig:RepComplexityConvHullZono}. The figure demonstrates that for high-dimensional zonotopes with a large number of generators, the Z-representation has the smallest representation complexity, even though we used lower bounds for the representation complexity in V- and H-representation.\\

\noindent For general polytopes, the V- and H-representation are usually more compact than the Z-representation.

\begin{figure}
	\centering
  	\includegraphics[width = 0.95 \textwidth]{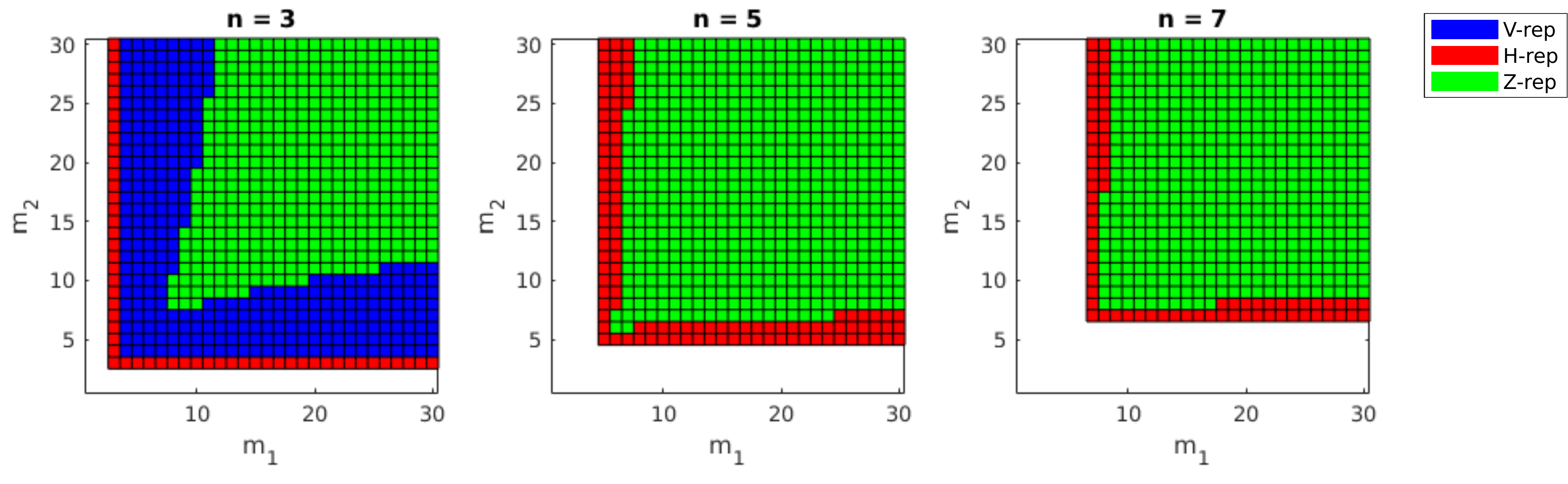}
	\captionof{figure}{Representation with the smallest representation complexity for an $n$-dimensional polytope $\mathcal{C}$ resulting from the convex hull of a zonotope with $m_1$ generators and a zonotope with $m_2$ generators.}
	\label{fig:RepComplexityConvHullZono}
\end{figure}


\section{Application: Range Bounding}
\label{sec:RangeBounding}

One application of the introduced Z-representation of polytopes is range bounding of nonlinear functions, which is defined as follows:

\begin{definition}
	(Range Bounding) Given a function $f: \mathbb{R}^{n} \to \mathbb{R}$ and a set $\mathcal{S} \subset \mathbb{R}^n$, the range bounding operation
	\begin{equation*}
		\operator{bound}(f(x),\mathcal{S}) \supseteq \left[\min\limits_{x \in \mathcal{S}} f(x),~ \max\limits_{x \in \mathcal{S}} f(x)\right]
	\end{equation*}
returns an over-approximation of the exact bounds.
\end{definition}

Two common approaches for range bounding are interval arithmetic \cite{Jaulin2006} and affine arithmetic \cite{deFigueiredo2004}. But due to the dependency problem \cite{Jaulin2006}, these techniques often result in large over-approximations. In addition, both approaches require the over-approximation of the set $\mathcal{S}$ by an axis-aligned interval, which further increases the conservatism in the calculated bounds. One can also use Taylor models \cite{Makino2003} for range bounding, which often enable the computation of significantly tighter bounds \cite{Althoff2018}. However, the Taylor model approach still requires over-approximating the set $\mathcal{S}$ by an axis-aligned interval. Due to the similarity between polynomial zonotopes and Taylor models, the principles behind the Taylor model approach for range bounding can easily be transferred to polynomial zonotopes, and consequently also to polytopes in Z-representation. The main advantage of using polynomial zonotopes for range bounding is that it is not required to over-approximate the set $\mathcal{S}$ with an axis-aligned interval. This enables the computation of significantly tighter function bounds, as the following example demonstrates:

\begin{example}
	We consider the nonlinear function
	\begin{equation*}
		f(x_1,x_2) = - (x_1-1.5)^2 - (x_2-1)^2 + 4 \cos(x_1) \sin(x_2)
	\end{equation*}
	and the polytope in Z-representation
	\begin{equation*}
		\mathcal{P} = \left \langle \begin{bmatrix} 0 \\ -0.5 \end{bmatrix}, \begin{bmatrix} 1 & 0 & 1 \\ -0.5 & 1.5 & -0.5 \end{bmatrix} ,~ \left( 1, 2, \begin{bmatrix} 1 \\ 2 \end{bmatrix} \right) \right\rangle_Z.
	\end{equation*}
	The function and the polytope are visualized in Fig. \ref{fig:RangeBounding}. For comparison, we determined the under-\linebreak[4]approximation $[-14.8872,1.4094]$ of the bounds with sampling. If interval arithmetic is applied to the interval over-approximation $\mathcal{I} \supseteq \mathcal{P}$ of the polytope, we obtain $\operator{bound}(f(x_1,x_2),\mathcal{I}) = [-25.25, 4]$. Applying the Taylor model approach for the interval over-approximation yields $\operator{bound}(f(x_1,x_2),\mathcal{I}) = [-19.7365, 2.3049]$. With the Z-representation, we can use the Taylor model approach directly for the polytope instead of the interval over-approximation resulting in $\operator{bound}(f(x_1,x_2),\mathcal{P}) = [-14.888, 1.4097]$, which is very close to the exact bounds.
	\label{ex:RangeBounding}
\end{example}

\begin{figure}
	\centering
  	\includegraphics[width = 0.9 \textwidth]{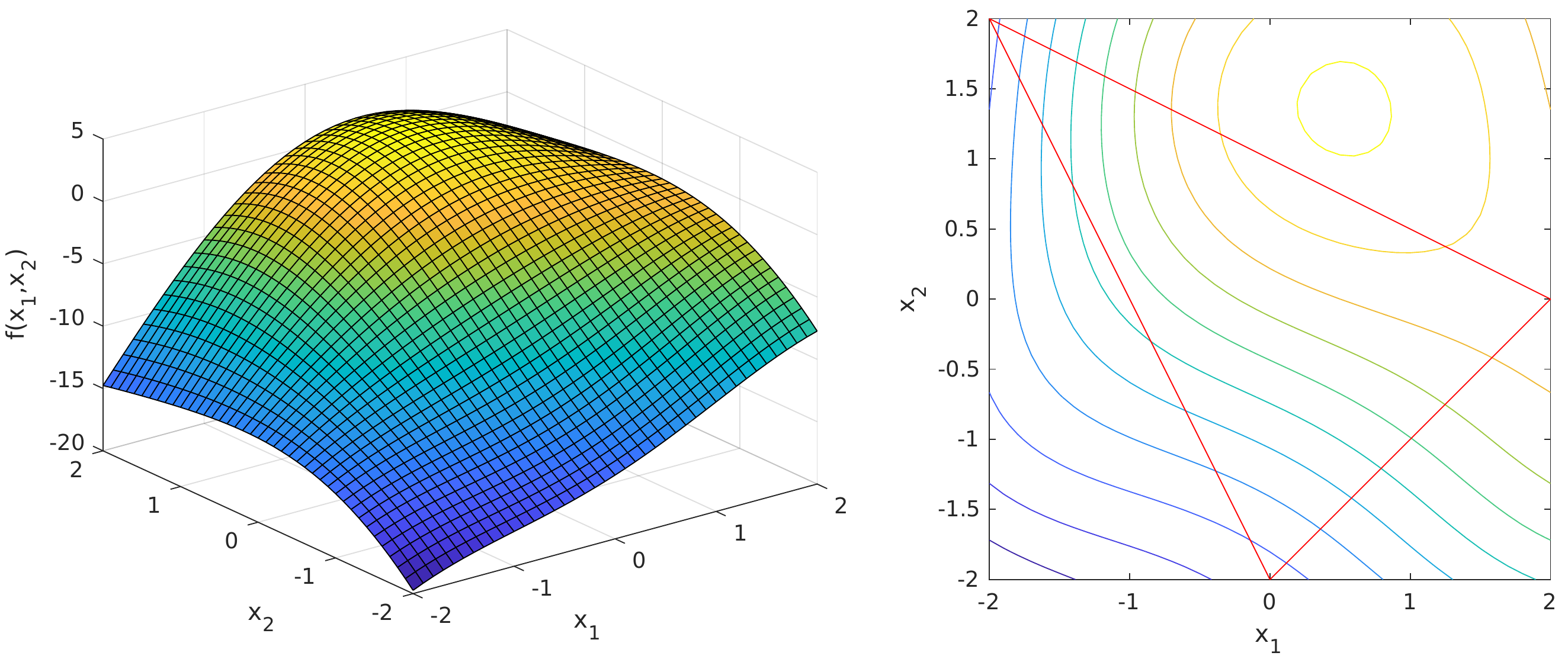}
	\captionof{figure}{Visualization of the function $f(x_1,x_2)$ (left) and the polytope $\mathcal{P}$ (red, right) from Example \ref{ex:RangeBounding}. In addition, the figure on the right shows isoclines of the function $f(x_1,x_2)$.}
	\label{fig:RangeBounding}
\end{figure}


\section{Conclusion}
\label{sec:Conclusion}

In this work, we introduced the novel Z-representation of bounded polytopes. Contrary to all other known representations of polytopes, this new representation enables the computation of linear transformation, Minkowski sum, and convex hull with a computational complexity that is polynomial in the representation size. We further specified algorithms for the conversion between Z-representation and other polytope representations. The complexity of the conversion from the V-representation to Z-representation is polynomial in the number of polytope vertices, and the conversion from Z-representation to V-representation has exponential complexity. In addition, we show that depending on the polytope, the Z-representation can be more compact than the H- and the V-representation. One application for the Z-representation is range bounding, which enables tighter bounds as demonstrated by a numerical example.

\begin{acknowledgements}
The authors gratefully acknowledge financial support by the German Research Foundation (DFG) project faveAC under grant AL 1185/5\_1. 
\end{acknowledgements}

\bibliographystyle{spmpsci}      
\bibliography{kochdumper,cpsGroup}   

\end{document}